\documentclass[11pt,reqno,final]{amsart}

\usepackage{graphicx}
\usepackage{amsmath}
\usepackage{epstopdf} 

\usepackage{mathptmx}      
\usepackage{url,graphicx,tabularx,array,geometry, amsthm, amsfonts,parskip,alltt,hyperref,xcolor}
\usepackage{mathtools,tikz,algorithm,algpseudocode,amssymb}
\usepackage{pgf,tikz}
\usepackage{mathrsfs}
\usetikzlibrary{arrows}

\newcommand{\ve}[1]{\mathbf{#1}}
\DeclareMathOperator{\argmax}{argmax}
\DeclareMathOperator{\R}{\mathbb{R}}

\newtheorem{lemma}{Lemma}
\newtheorem{corollary}{Corollary}

\begin{document}

\title{On Motzkin's Method for Inconsistent Linear Systems}

\author{Jamie Haddock         \and
        Deanna Needell 
}

\begin{abstract}
Iterative linear solvers have gained recent popularity due to their computational efficiency and low memory footprint for large-scale linear systems. The \textit{relaxation method}, or \textit{Motzkin's method}, can be viewed as an iterative method that projects the current estimation onto the solution hyperplane corresponding to the most violated constraint. Although this leads to an optimal selection strategy for consistent systems, for inconsistent least square problems, the strategy presents a tradeoff between convergence rate and solution accuracy. 
We provide a theoretical analysis that shows Motzkin's method offers an initially accelerated convergence rate and this acceleration depends on the dynamic range of the residual.  We quantify this acceleration for Gaussian systems as a concrete example. 
Lastly, we include experimental evidence on real and synthetic systems that support the analysis. 
\end{abstract}

\thanks{This material is based upon work supported by the National Science Foundation under Grant No. DMS-1440140 while the authors were in residence at the Mathematical Sciences Research Institute in Berkeley, California, during the Fall 2017 semester. JH was also partially supported by NSF grant DMS-1522158 and the University of California, Davis Dissertation Fellowship. DN was also supported by NSF CAREER award $\#1348721$ and NSF BIGDATA $\#1740325$.
}

\maketitle

\section{Introduction}
We consider solving large-scale systems of linear equations represented by a matrix $A\in\mathbb{R}^{m\times n}$ and vector $\ve{b}\in\mathbb{R}^m$; we use the convention that vectors are bold type, and matrices and scalars are not. We are interested in the highly overdetermined setting, where $m \gg n$, which means the system need not necessarily have a solution.  Iterative solvers like the Kaczmarz method \cite{Kac37:Angenaeherte-Aufloesung,strohmer2009randomized}, Motzkin's method \cite{motzkinschoenberg,agmon,SKM}, and the Gauss-Seidel method \cite{LL10:Randomized-Methods,ma2015convergence}  have become re-popularized recently for such problems since they are particularly efficient in terms of computation and storage.

The Kaczmarz method is a popular iterative solver for overdetermined systems of linear equations and is especially preferred for large-scale systems since it need not ever load the entire system into memory at once.  The method consists of sequential orthogonal projections toward the solution set of a single equation (or subsystem).  Given the system $A\ve{x} = \ve{b}$, the method computes iterates by projecting onto the hyperplane defined by the equation $\ve{a}_i^T \ve{x} = b_i$ where $\ve{a}_i^T$ is a selected row of the matrix $A$ and $b_i$ is the corresponding entry of $\ve{b}$.  The iterates are recursively defined as 
\begin{equation}\label{rk}
\ve{x}_{k+1} = \ve{x}_k + \frac{b_i - \ve{a}_i^T\ve{x}_k}{\|\ve{a}_i\|^2} \ve{a}_i
\end{equation} 
where $\ve{a}_i^T$ is selected from among the rows of $A$ (and the initialization $\ve{x}_0$ is chosen arbitrarily).  The seminal work of Strohmer and Vershynin \cite{strohmer2009randomized} proved exponential convergence for the randomized Kaczmarz method where the $i$th row $\ve{a}_i^T$ is chosen with probability $\|\ve{a}_i\|^2/\|A\|_F^2$.  Since then many variants of the method have been proposed and analyzed for various types of systems, see e.g., \cite{REK,Petra2016,GreedyKaczmarz,RefWorks:498,gower2015randomized,RefWorks:295,richtarik2012iteration} and references therein. 

It is known that the randomized Kaczmarz method converges for inconsistent systems (or equivalently those corrupted by noise) with an error threshold dependent on $A$ and the noise. In \cite{deanna} it was shown that this method has iterates that satisfy:
\begin{equation}\label{RKits}
\mathbb{E}\|\ve{x}_k - \ve{x}_{\text{LS}}\|^2 \leq \left(1 - \frac{\sigma_{\text{min}}^2(A)}{\|A\|_F^2}\right)^k \|\ve{x}_0 -  \ve{x}_{\text{LS}}\|^2 + \frac{\|A\|_F^2}{\sigma_{\text{min}}^2(A)}\|\ve{e}\|_{\infty}^2,
\end{equation}
where here and throughout, the norm without subscript, $\|\cdot\|$, denotes the Euclidean norm, $\sigma_{\text{min}}(A)$ denotes the minimum singular value of $A$, $\|A\|_F$ its Frobenius norm, $ \ve{x}_{\text{LS}}$ the least squares solution and $\ve{e} = \ve{b} - A \ve{x}_{\text{LS}}$ denotes the error term.
There are variants of this method that converge to the least squares solution, e.g. \cite{REK} that utilizes an additional projection step to project off the error term $\ve{e}$.  Additionally, it is known that if a linear system of equations or inequalities is feasible then randomized Kaczmarz will provide a \emph{proof} or \emph{certificate of feasibility}, and there are probabilistic guarantees on how quickly it will do so \cite{SKM}.   

A related but seemingly disjointedly studied work is an approach by Agmon \cite{agmon}, and Motzkin and  Schoenberg \cite{motzkinschoenberg}, re-imagined a few years later as the now famous \emph{perceptron} algorithm \cite{rosenblatt}. These approaches are most often used for feasibility problems, where one seeks a point that resides within some polyhedron described by a system of inequalities; of course, linear systems of equations are one special instance.  Additionally, this so-called Motzkin method has been referred to as the Kaczmarz method with the ``most violated constraint'' or ``maximal-residual'' control \cite{CensorRowAction,GreedyKaczmarz,Petra2016}. As these descriptors suggest, this method iterates in a similar fashion as the Kaczmarz method, but rather than selecting a row of $A$ in sequential or randomized order, it selects the row corresponding to the most violated constraint, as described in Algorithm \ref{alg:MotzSelect}. Starting from any initial point $\ve{x}_0$, the method proceeds as follows.
If the current point $\ve{x}_k$ is a solution, the method terminates; otherwise there must be a constraint $\ve{a}_i^T \ve{x}  = b_i$ that is \emph{most violated}. The constraint defines a hyperplane $H$. The method then projects $\ve{x}_k$ onto this hyperplane as in \eqref{rk}, or perhaps under/over projects using an alternate step-size, see \cite{motzkinschoenberg,SKM} for details. Selecting the most violated constraint is intuitive for feasibility problems or for solving consistent linear systems of equations. In the inconsistent case, it may not always make sense to project onto the most violated constraint; see Figure \ref{motzkinpic} for a simple example of this situation.  However, following \cite{SKM}, we present experimental evidence that suggests Motzkin's method often offers an initially accelerated convergence rate, both for consistent and inconsistent systems of equations.  

\begin{algorithm}
	\caption{Motzkin method (for normalized $A$)}\label{alg:MotzSelect}
	\begin{algorithmic}[1]
		\Procedure{Motzkin}{$A,\ve{b},\ve{x}_0,k,$}
		\For{$j=1,2,...,k$}
		\State $\ve{x}_j = \ve{x}_{j-1} +  (b_{i_j} - \ve{a}_{i_j}^T \ve{x}_{j-1}) \ve{a}_{i_j}$ where $i_j = \underset{i \in [m]}{\argmax} \;\;(\ve{a}_i^T\ve{x}_{j-1} - b_i)^2$.
		\EndFor
		\State \textbf{return} $\ve{x}_k$
		\EndProcedure
	\end{algorithmic}
\end{algorithm}

\subsection{Contribution}
We show that Motzkin's method for systems of linear equations features an initially accelerated convergence rate when the residual has a large dynamic range. We provide bounds for the iterate error which depend on the dynamic range of the residual. These bounds can potentially be used when designing stopping criteria or hybrid approaches.  Next, for a concrete example we show that Gaussian systems of linear equations have large dynamic range and provide bounds on this value.  We extend this to a corollary which shows that the initial convergence rate is highly accelerated and our theoretical bound closely matches experimental evidence.

\section{Accelerated Convergence of Motzkin's Method}

The advantage of the Motzkin method is that by greedily selecting the most violated constraint, the method makes large moves at each iteration, thereby accelerating convergence. One drawback of course, is that it is computationally expensive to compute which constraint is most violated. For this reason, De Loera et al. \cite{SKM} proposed a hybrid batched variant of the method that randomly selects a batch of rows and then computes the most violated from that batch.  This method is quite fast when using parallel computation, but the method often offers accelerated convergence that outweighs the increased computational cost even without parallelization techniques. When the system is inconsistent, however, there is an additional drawback to the Motzkin method because projecting onto the most violated constraint need not move the iterate closer to the desired solution, as already mentioned and shown in Figure \ref{motzkinpic}.  {Our first lemma provides a rule for deciding if a greedy projection offers desirable improvement.}  Here and throughout, we assume that the matrix $A$ has been normalized to have unit row norm, $\|\ve{a}_i\|^2 = 1$, and that the matrix has full column rank, $n$.

\begin{figure}
	\begin{center}
		\definecolor{uuuuuu}{rgb}{0.26666666666666666,0.26666666666666666,0.26666666666666666}
		\definecolor{ttqqqq}{rgb}{0.2,0.,0.}
		\definecolor{qqqqff}{rgb}{0.,0.,1.}
		\begin{tikzpicture}[line cap=round,line join=round,>=triangle 45,x=1.0cm,y=1.0cm]
		\clip(-2.72202174677996,-0.7094879538883636) rectangle (6.994223122074292,5.228576774309162);
		\draw [domain=-5.72202174677996:8.994223122074292] plot(\x,{(--2.1924--1.82*\x)/1.32});
		\draw [domain=-5.72202174677996:8.994223122074292] plot(\x,{(--4.1776--1.04*\x)/2.06});
		\draw [domain=-5.72202174677996:8.994223122074292] plot(\x,{(--4.7488--0.64*\x)/2.24});
		\draw [domain=-5.72202174677996:8.994223122074292] plot(\x,{(-3.0668--4.1*\x)/4.92});
		\draw (2.88,2.72)-- (1.4960971749624123,3.7237097412360525);
		\draw (1.4960971749624123,3.7237097412360525)-- (3.020832631733253,1.8940271931110435);
		\draw (3.020832631733253,1.8940271931110435)-- (2.8625635050027434,2.0088157905199844);
		\draw (2.724983287542531,2.108599244941677)-- (2.5820314794169668,2.2122785783074708);
		\draw (2.4479104538511947,2.3095531682782506)-- (2.3100771247889496,2.4095201981475713);
		\draw (2.8537993169502803,3.0472093384145627) node[anchor=north west] {$\ve{x}_0$};
		\draw (0.9514545491861374,4.129924818267152) node[anchor=north west] {$\ve{x}_1$};
		\draw (3.0704098151186976,2.137647452700909) node[anchor=north west] {$\ve{x}_2$};
		\draw (0.0385368627398501,2.7238872387808794) node[anchor=north west] {$\ve{x}^*$};
		\begin{scriptsize}
		\draw [fill=qqqqff] (0.42,2.24) circle (2.0pt);
		\draw [fill=ttqqqq] (2.88,2.72) circle (1.5pt);
		\draw [fill=uuuuuu] (1.4960971749624123,3.7237097412360525) circle (1.5pt);
		\draw [fill=uuuuuu] (3.020832631733253,1.8940271931110435) circle (1.5pt);
		\end{scriptsize}
		\end{tikzpicture}
	\end{center}
	\caption{An example of a series of projections using the Motzkin approach on an inconsistent system.  Lines represent the hyperplanes consisting of sets $\{\ve{x}: \ve{a}_i^T \ve{x} = b_i\}$ for rows $\ve{a}_i^T$ of $A$, and $\ve{x}^*$ denotes the desired solution.}\label{motzkinpic}
\end{figure}
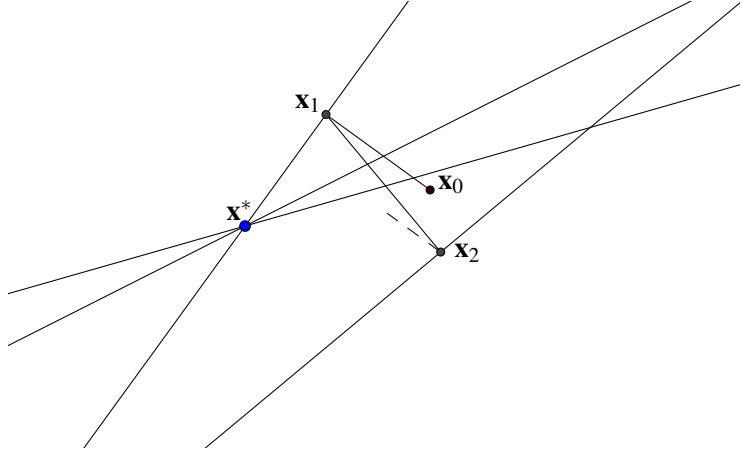

\begin{lemma}\label{lem:motzkin}
	Let $\ve{x}$ denote any desired solution of the system given by matrix $A$ and right hand side $\ve{b}$. If $\ve{e} = A\ve{x}- \ve{b}$ and $\|A\ve{x}_k - \ve{b}\|_\infty > 4\|\ve{e}\|_\infty$ then the next iterate, $\ve{x}_{k+1}$ defined by Algorithm \ref{alg:MotzSelect} satisfies 
	$$\|\ve{x}_{k+1} - \ve{x}\|^2 \le \|\ve{x}_k - \ve{x}\|^2 - \frac{1}{2} \|A\ve{x}_k - \ve{b}\|_\infty^2.$$
\end{lemma}

\begin{proof}
	By definition of $\ve{x}_{k+1}$, we have
	\begin{align}\label{arg}
	\|\ve{x}_{k+1} - \ve{x}\|^2 &= \|\ve{x}_k - \ve{x}\|^2 - 2(\ve{a}_{i_{k+1}}^T\ve{x}_k-b_{i_{k+1}})(\ve{a}_{i_{k+1}}^T\ve{x}_k - b_{i_{k+1}} - e_{i_{k+1}}) + (\ve{a}_{i_{k+1}}^T\ve{x}_k - b_{i_{k+1}})^2 \notag
	\\&= \|\ve{x}_k - \ve{x}\|^2 - (\ve{a}_{i_{k+1}}^T\ve{x}_k - b_{i_{k+1}})^2 + 2(\ve{a}_{i_{k+1}}^T\ve{x}_k-b_{i_{k+1}})e_{i_{k+1}}\notag
	\\&\le \|\ve{x}_k - \ve{x}\|^2 - (\ve{a}_{i_{k+1}}^T\ve{x}_k - b_{i_{k+1}})^2 + 2|\ve{a}_{i_{k+1}}^T\ve{x}_k-b_{i_{k+1}}|\cdot|e_{i_{k+1}}|\notag
	\\&= \|\ve{x}_k - \ve{x}\|^2 - \|A\ve{x}_k - \ve{b}\|_\infty^2 + 2 \|A\ve{x}_k - \ve{b}\|_\infty |e_{i_{k+1}}|\notag
	\\&\le \|\ve{x}_k - \ve{x}\|^2 - \|A\ve{x}_k - \ve{b}\|_\infty^2 + 2 \|A\ve{x}_k - \ve{b}\|_\infty \|\ve{e}\|_\infty
	\\&\le \|\ve{x}_k - \ve{x}\|^2 - \frac{1}{2} \|A\ve{x}_k - \ve{b}\|_\infty^2.\notag
	\end{align}
\end{proof}

Note that this tells us that while our residual is still large relative to the error, Motzkin's method can offer good progress in each iteration.  Also, this progress is better than the expected progress offered by Randomized Kaczmarz (RK) when the residual has good dynamic range, in particular when:
$$\frac{1}{2} \|A\ve{x}_k - \ve{b}\|_\infty^2 > \frac{1}{m} \|A\ve{x}_k - \ve{b}\|^2.$$  We can use Lemma \ref{lem:motzkin} to easily obtain the following corollary.

\begin{corollary}\label{cor:Motzrate}
	Let $\ve{x}$ denote any desired solution of the system given by matrix $A$ and right hand side $\ve{b}$ and write $\ve{e} = A\ve{x}-\ve{b}$ as the error term. 
	Then for any given iteration $k$, the iterate defined by Algorithm \ref{alg:MotzSelect} satisfies either (i) or both (ii) and (iii), where
	\begin{align*}
	&\text{(i)}\quad \|\ve{x}_{k+1} - \ve{x}\|^2 \le  \|\ve{x}_k - \ve{x}\|^2 - \frac{1}{2} \|A\ve{x}_k - \ve{b}\|_\infty^2\\
	&\text{(ii)}\quad \|\ve{x}_k - \ve{x}\|^2 \leq 25m\sigma_{\min}^{-2}(A)\|\ve{e}\|_{\infty}^2\\
	&\text{(iii)}\quad \|\ve{x}_{k+1} - \ve{x}\|^2 \le \left(25m\sigma_{\min}^{-2}(A) + 8\right)\|\ve{e}\|_{\infty}^2.
	\end{align*}
	In addition, if the method is run for $K$ iterations with the stopping criterion $\|A\ve{x}_K - \ve{b}\|_\infty \leq 4\|\ve{e}\|_\infty$, then  
	the method exhibits the (possibly highly accelerated) convergence rate
	\begin{align}
	\|\ve{x}_{K} - \ve{x}\|^2 &\le \prod_{k=0}^{K-1}\left(1 - \frac{\sigma_{\min}^2(A)}{4\gamma_k}\right)\cdot\|\ve{x}_0 - \ve{x}\|^2 + 2m\sigma_{\min}^{-2}(A)\|\ve{e}\|_{\infty}^2,\label{cor2a}\\
	&\le \left(1 - \frac{\sigma_{\min}^2(A)}{4m}\right)^K\|\ve{x}_0 - \ve{x}\|^2 + 2m\sigma_{\min}^{-2}(A)\|\ve{e}\|_{\infty}^2,\label{cor2}
	\end{align}
	with final error satisfying (ii). 
	Here $\gamma_k$ bounds the dynamic range of the $k$th residual, $\gamma_k := \frac{\|A\ve{x}_k-A\ve{x}\|^2}{\|A\ve{x}_k-A\ve{x}\|_{\infty}^2}$.
\end{corollary}

\begin{proof}
	We consider two cases, depending on whether $\|A\ve{x}_k -\ve{b}\|_\infty > 4\|\ve{e}\|_\infty$ or $\|A\ve{x}_k - \ve{b}\|_\infty \leq 4\|\ve{e}\|_\infty$. If the former holds, then (i) is valid by Lemma \ref{lem:motzkin}.  If instead the latter holds, then we first obtain (ii) by the simple argument
	\begin{align*}
	\|\ve{x}_{k} - \ve{x}\|^2 &\leq \sigma_{\min}^{-2}(A)\|A\ve{x}_k - A\ve{x}\|^2\\
	&\leq  \sigma_{\min}^{-2}(A)m\|A\ve{x}_{k} - A\ve{x}\|_{\infty}^2\\  
	&\leq \sigma_{\min}^{-2}(A)m\left(\|A\ve{x}_{k} - \ve{b}\|_{\infty}^2 + 2\|A\ve{x}_k - \ve{b}\|_\infty \|\ve{e}\|_\infty + \|\ve{e}\|_{\infty}^2\right)\\
	&\leq \sigma_{\min}^{-2}(A)m\left(16\|\ve{e}\|_{\infty}^2 + 8\|\ve{e}\|_{\infty}^2 + \|\ve{e}\|_{\infty}^2\right)\\
	&= 25m\sigma_{\min}^{-2}(A)\|\ve{e}\|_{\infty}^2.
	\end{align*}
	To obtain (iii) still in this latter case, we continue from \eqref{arg} showing
	\begin{align*}
	\|\ve{x}_{k+1} - \ve{x}\|^2 &\leq \|\ve{x}_k - \ve{x}\|^2 - \|A\ve{x}_k - \ve{b}\|_\infty^2 + 2 \|A\ve{x}_k - \ve{b}\|_\infty \|\ve{e}\|_\infty\\
	&\leq 25m\sigma_{\min}^{-2}(A)\|\ve{e}\|_{\infty}^2  - \|A\ve{x}_k - \ve{b}\|_\infty^2 + 2 \|A\ve{x}_k - \ve{b}\|_\infty \|\ve{e}\|_\infty\\
	&\leq 25m\sigma_{\min}^{-2}(A)\|\ve{e}\|_{\infty}^2  + 2 \|A\ve{x}_k - \ve{b}\|_\infty \|\ve{e}\|_\infty\\
	&\leq 25m\sigma_{\min}^{-2}(A)\|\ve{e}\|_{\infty}^2  + 8 \|\ve{e}\|_\infty^2\\
	&= \left(25m\sigma_{\min}^{-2}(A)+ 8 \right) \|\ve{e}\|_\infty^2.
	\end{align*}
	To prove \eqref{cor2a} and \eqref{cor2}, we first note that by choice of stopping criterion, (i) holds for all $0\leq k \leq K$. Thus for all such $k$, we have
	\begin{align}
	\|\ve{x}_{k} - \ve{x}\|^2 &\leq \|\ve{x}_{k-1} - \ve{x}\|^2 - \frac{1}{2}\|A\ve{x}_{k-1} - \ve{b}\|_\infty^2\notag\\
	&= \|\ve{x}_{k-1} - \ve{x}\|^2 - \frac{1}{2}\|(A\ve{x}_{k-1} - A\ve{x}) - \ve{e}\|_\infty^2\notag\\
	&\leq \|\ve{x}_{k-1} - \ve{x}\|^2 - \frac{1}{4}\|A\ve{x}_{k-1} - A\ve{x}\|_\infty^2 + \frac{1}{2}\|\ve{e}\|_{\infty}^2 \label{loser}\\
	&= \|\ve{x}_{k-1} - \ve{x}\|^2 - \frac{1}{4\gamma_{k-1}}\|A\ve{x}_{k-1} - A\ve{x}\|^2 + \frac{1}{2}\|\ve{e}\|_{\infty}^2 \notag\\
	&\leq \|\ve{x}_{k-1} - \ve{x}\|^2 - \frac{\sigma_{\min}^{2}(A)}{4\gamma_{k-1}}\|\ve{x}_{k-1} - \ve{x}\|^2 + \frac{1}{2}\|\ve{e}\|_{\infty}^2\notag \\
	&= \left(1 - \frac{\sigma_{\min}^{2}(A)}{4\gamma_{k-1}}\right)\|\ve{x}_{k-1} - \ve{x}\|^2 + \frac{1}{2}\|\ve{e}\|_{\infty}^2,\label{last1}
	\end{align}
	where the first line follows from (i), the third from Jensen's inequality, and the fifth  from properties of singular values.

	Iterating the relation given by \eqref{last1} recursively yields\footnote{We use the convention that an empty sum or product equates to one.}
	\begin{align*}
	\|\ve{x}_{K} - \ve{x}\|^2 &\leq \prod_{k=0}^{K-1}\left(1 - \frac{\sigma_{\min}^{2}(A)}{4\gamma_k}\right)\cdot\|\ve{x}_{0} - \ve{x}\|^2 + \sum_{j=0}^{K-1} \prod_{k=0}^{j-1}\left(1 - \frac{\sigma_{\min}^{2}(A)}{\gamma_k}\right)\frac{1}{2}\|\ve{e}\|_{\infty}^2\\
	&\leq \prod_{k=0}^{K-1}\left(1 - \frac{\sigma_{\min}^{2}(A)}{4\gamma_k}\right)\cdot\|\ve{x}_{0} - \ve{x}\|^2 + \sum_{j=0}^{K-1}\left(1 - \frac{\sigma_{\min}^{2}(A)}{4m}\right)^j\frac{1}{2}\|\ve{e}\|_{\infty}^2\\
	&\leq  \prod_{k=0}^{K-1}\left(1 - \frac{\sigma_{\min}^{2}(A)}{4\gamma_k}\right)\cdot\|\ve{x}_{0} - \ve{x}\|^2 + 2m\sigma_{\min}^{-2}(A)\|\ve{e}\|_{\infty}^2\\
	&\leq  \left(1 - \frac{\sigma_{\min}^{2}(A)}{4m}\right)^K\|\ve{x}_{0} - \ve{x}\|^2 + 2m\sigma_{\min}^{-2}(A)\|\ve{e}\|_{\infty}^2,
	\end{align*}
	where the second and fourth inequalities follow from the simple bound $\gamma_k \leq {m}$ and the third by bounding above by the infinite sum. The last two inequalities complete the proof of \eqref{cor2a} and \eqref{cor2}.
\end{proof}

Note that Lemma \ref{lem:motzkin} and Corollary \ref{cor:Motzrate} are true for any desired solution, $\ve{x}$.  Here the desired solution could be the least squares solution or generally any other point.  However, the residual of the desired solution, $A\ve{x} - \ve{b}$, determines the error $\ve{e}$ and the final error of Motzkin's method.

\begin{figure}
	\begin{center}
		\includegraphics[width=0.6\textwidth]{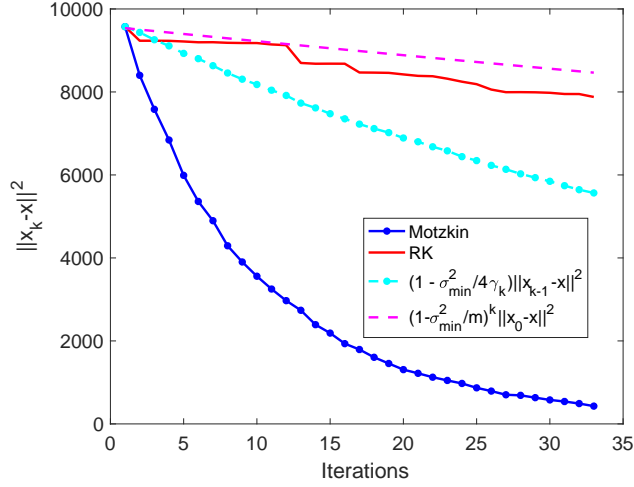}
		\caption{Convergence of Motzkin's method and RK on correlated system with corresponding theoretical bounds.}\label{fig:corrconvergence}
	\end{center}
\end{figure}

We note that the convergence rate given by \eqref{cor2a} yields a significant improvement over that given by \eqref{RKits} when the dynamic range of many residuals is large, i.e. when $\gamma_k \ll m$ for many iterations $k$.  In Figure \ref{fig:corrconvergence}, we present the convergence of Motzkin and RK on a random system which before normalization is defined by matrix $A \in \mathbb{R}^{5000 \times 100}$ with $a_{ij} \sim \mathcal{N}(1,0.5)$ and $\ve{b} = A\mathbf{1} + \ve{\epsilon}$ where $\mathbf{1}$ denotes the all ones vector and $\ve{\epsilon}$ is a Gaussian vector, and the corresponding theoretical bounds.  Figure \ref{fig:Netlib} presents plots providing the convergence of Motzkin and RK, and the corresponding theoretical bounds on systems of equations defined by problems from the \texttt{Netlib} linear programming benchmark set \cite{Netlib}. These problems contain naturally under-determined systems, which we transform into overdetermined, inconsistent systems with nearly the same least-squares solution. We transform the problem, originally given by the underdetermined systems of equations $A\ve{x} = \ve{b}$ by adding equations to form $$\begin{bmatrix} A \\ I \end{bmatrix} \ve{x} = \begin{bmatrix} \ve{b} \\ \ve{x}_{LS} + \ve{\epsilon} \end{bmatrix}$$ where $\ve{x}_{LS}$ is the least-norm solution of $A\ve{x} = \ve{b}$ and $\ve{\epsilon}$ is a Gaussian vector with small variance, and normalizing the resulting system.  Each problem has very small error which is distributed relatively uniformly, thus there are many iterations in which the theoretical bounds hold.  The resulting matrix for problem \texttt{agg} is of size $1103 \times 615$, the resulting matrix for problem \texttt{agg2} is of size $1274 \times 758$, the resulting matrix for problem \texttt{agg3} is also of size $1274 \times 758$, and the resulting matrix for problem \texttt{bandm} is of size $777 \times 472$. These plots are only for the iterations before the stopping criterion is met.

\begin{figure}[h]
	\includegraphics[width=0.5\textwidth]{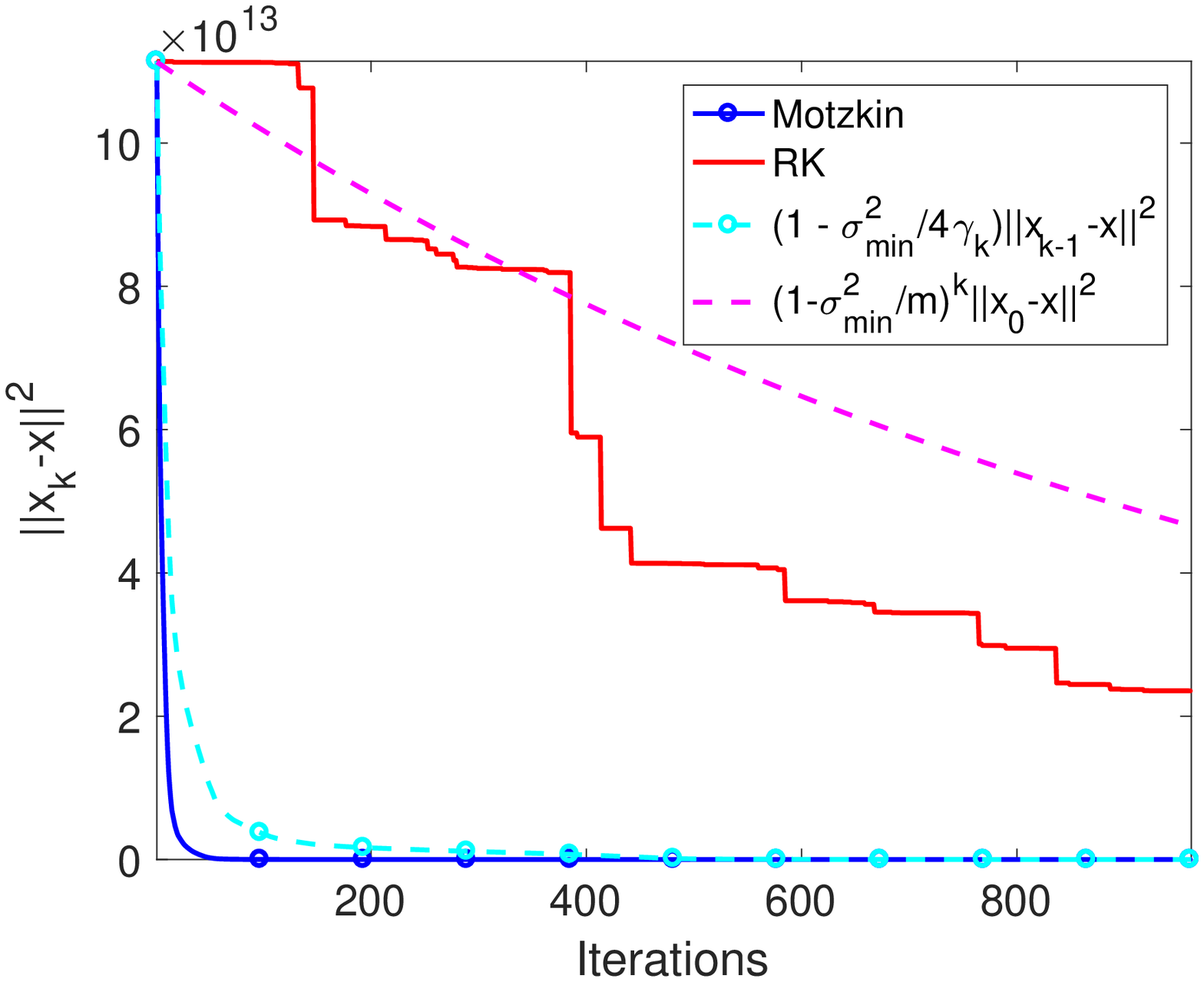}%
	\includegraphics[width=0.5\textwidth]{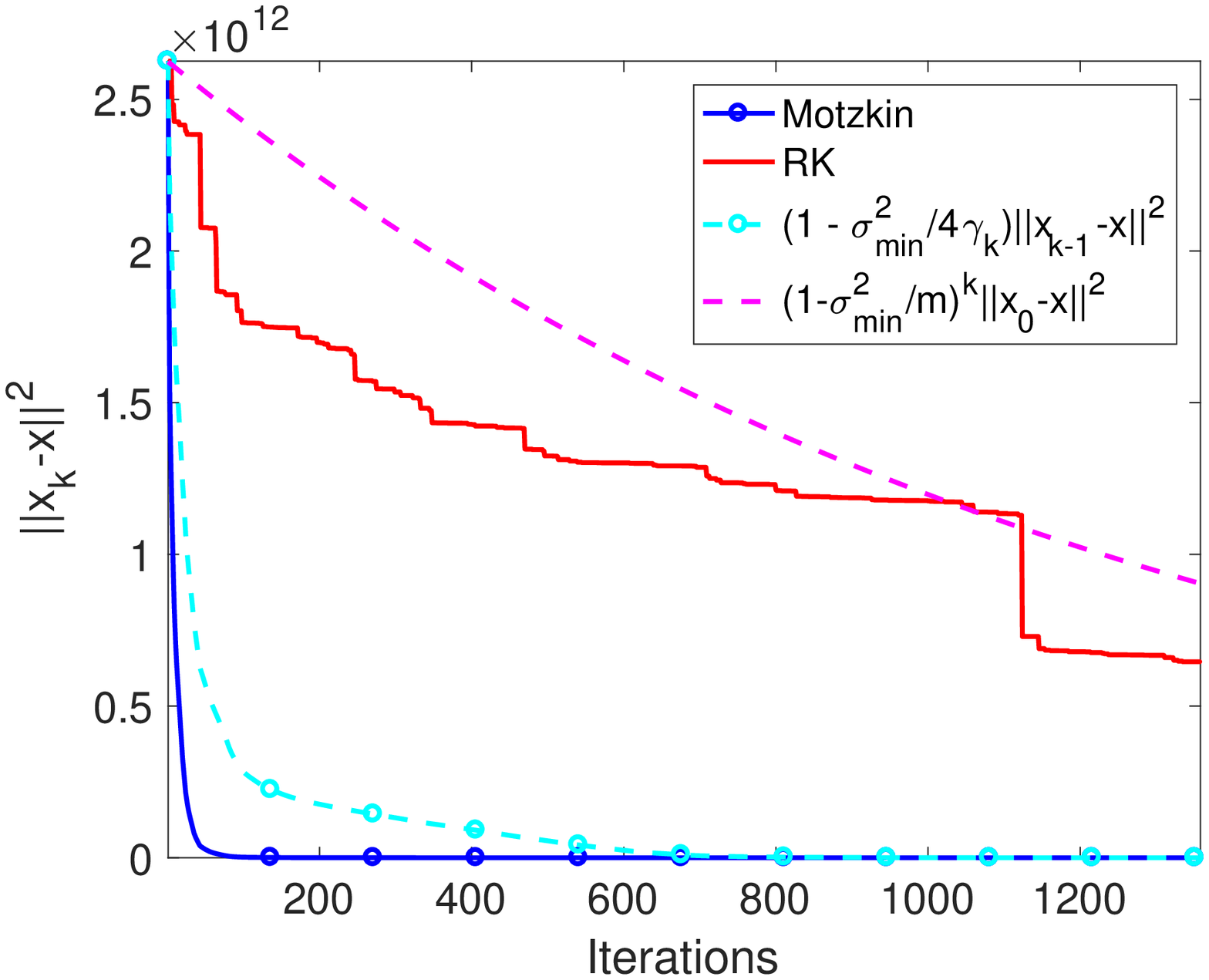}%
	\\\includegraphics[width=0.5\textwidth]{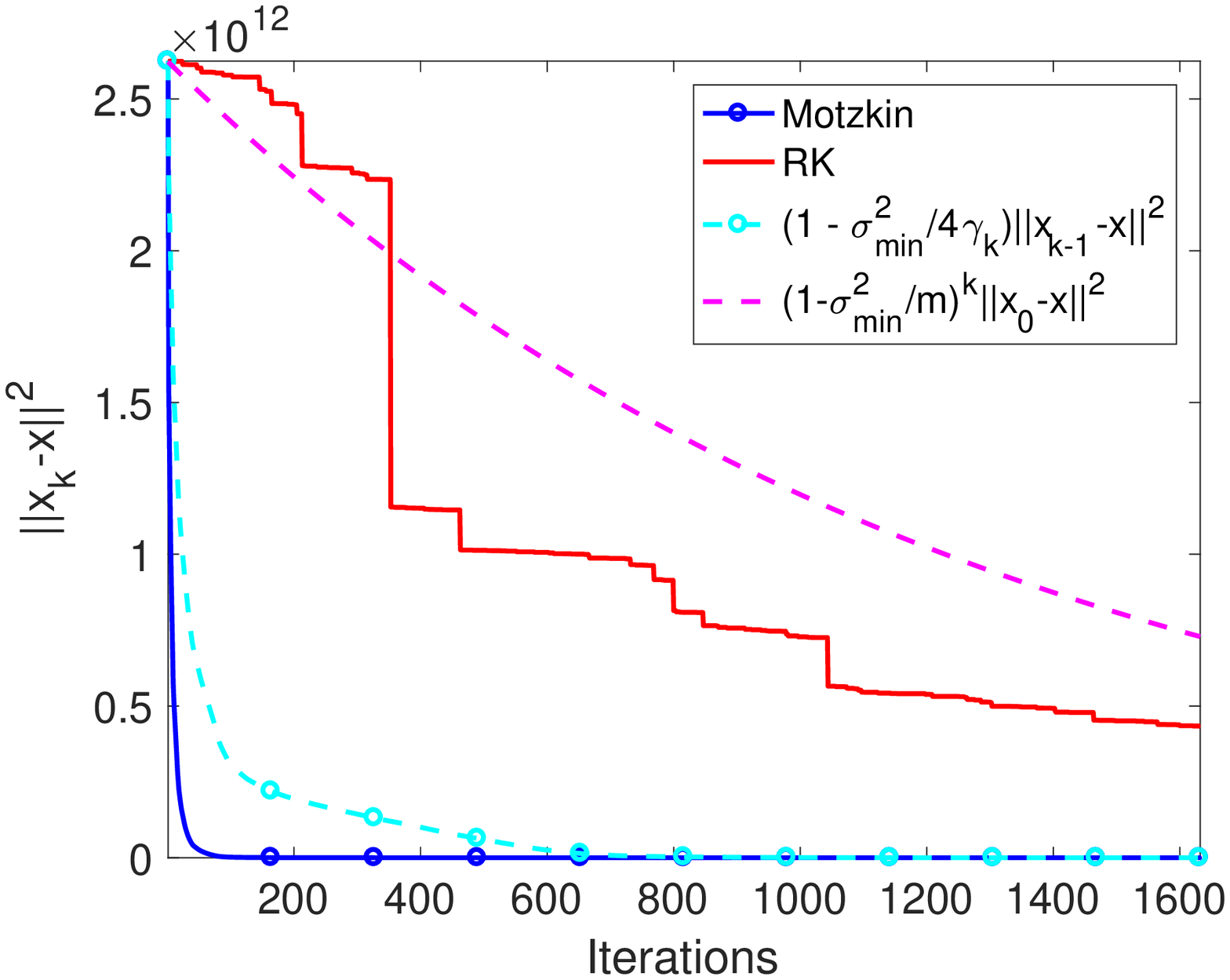}%
	\includegraphics[width=0.5\textwidth]{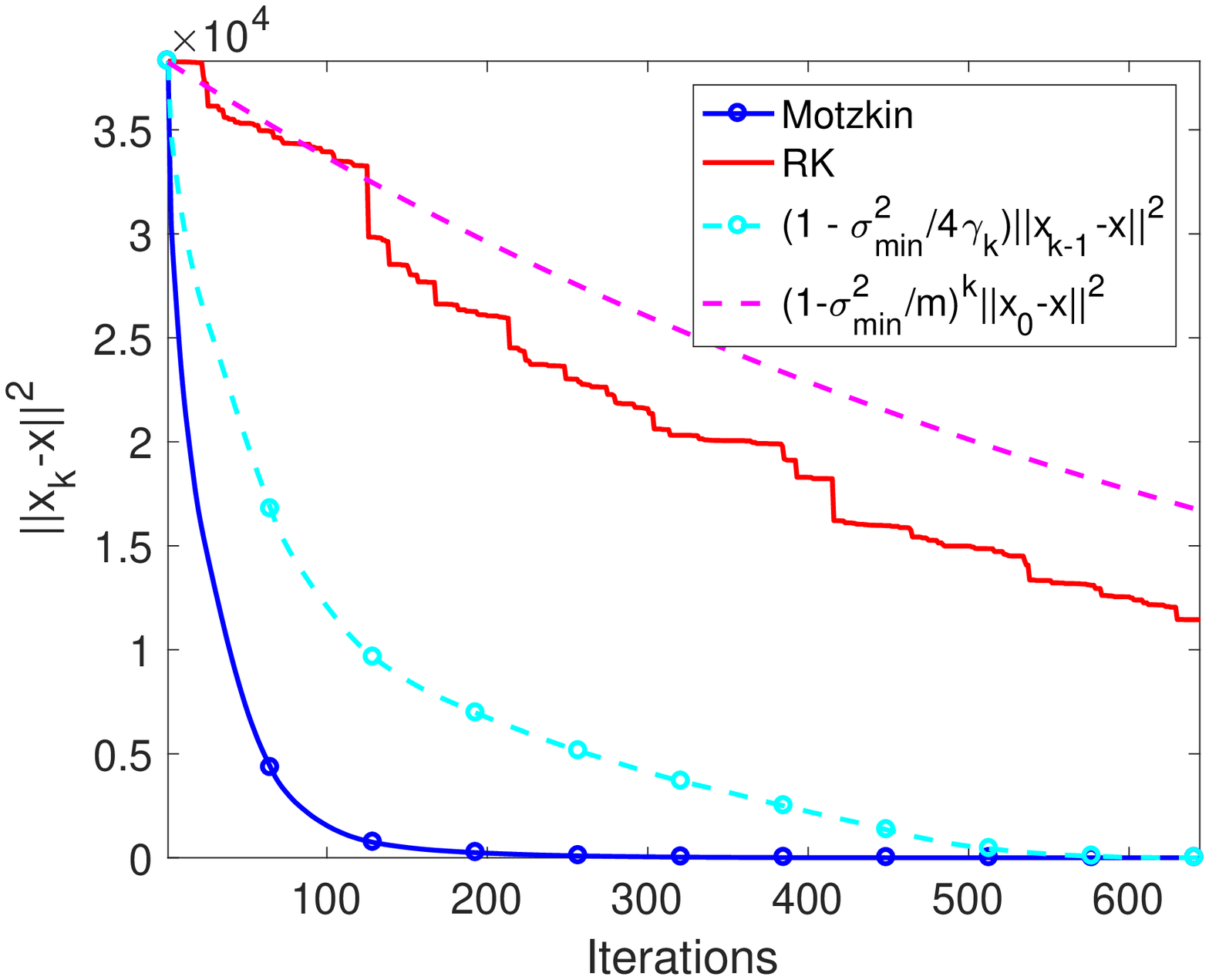}%
	\caption{Convergence of Motzkin's method and RK, and corresponding theoretical bounds for \texttt{Netlib} linear programming problems.  Upper left: \texttt{agg}; upper right: \texttt{agg2}; lower left: \texttt{agg3}; lower right: \texttt{bandm}.}
	\label{fig:Netlib}
\end{figure}  

In Table \ref{tab:Netlib}, we include the CPU computation time (computed with the Matlab function \texttt{cputime}) required to reach 
residual norm, $\|A\ve{x}_k-\ve{b}\|_\infty$, less than $4\|A\ve{x}_{\text{LS}} -\ve{b}\|_\infty$; that is to compute $\ve{x}_k$ with $\|A\ve{x}_k-\ve{b}\|_\infty \le 4 \|A\ve{x}_{\text{LS}}-\ve{b}\|_\infty$.  
We include computation times averaged over 10 trials for Motzkin's method and the Randomized Kaczmarz method on the \texttt{Netlib} problems \texttt{agg}, \texttt{agg2}, \texttt{agg3}, and \texttt{bandm}.  Note that this computation is performed with no parallelization implemented for Motzkin's method, which means that each iteration of Motzkin's method is much more costly than that of RK.  Nevertheless, Motzkin's method outperforms RK on some of the selected \texttt{Netlib} problems.  However, this is not the focus of this paper, as the acceleration described in Lemma \ref{lem:motzkin} does not necessarily guarantee Motzkin's method a computational advantage if the iterations are significantly more costly than those of RK.

This acceleration is in force until the stopping criterion given in the corollary.  This bound therefore, can be used to design such stopping criteria; one could design an approach for example that utilizes the Motzkin method until reaching this threshold, and then switching to the traditional RK selection strategy to reduce the convergence horizon.  
In Figure \ref{fig:MSGDvsRK}, we see that Motzkin outperforms RK for the initial iterations (while $\|A \ve{x}_k - \ve{b}\|_\infty \gg \|\ve{e}\|_\infty$) on a system with Gaussian noise.  Here, before normalization, the system consists of Gaussian matrix $A \in \mathbb{R}^{50000 \times 100}$ and right-hand side $\ve{b} = A\mathbf{1} + \ve{e}$ where $\mathbf{1}$ is the vector of all ones and $\ve{e}$ is a Gaussian vector. However, for a system with sparse, large magnitude error, Motzkin does not perform as well in the long run, as it suffers from a worse \textit{convergence horizon} than RK. Here, before normalization, the system consists of Gaussian matrix $A \in \mathbb{R}^{50000 \times 100}$ and right-hand side $\ve{b} = A\mathbf{1} + 15\sum_{j \in S} \ve{e}_j$ where $\ve{e}_j$ denotes the $j$th coordinate vector and $S$ is a uniform random sample of 50 indices.

\begin{table}
	\begin{center}
		\begin{tabular}{c || c || c | c}
			Problem & $4\|A\ve{x}_{\text{LS}}-\ve{b}\|_\infty$ & Motzkin (s) & RK (s) \\
			\hline
			\texttt{agg} & $2.16*10^{-8}$ & 0.723 & 0.836 \\ 
			\texttt{agg2} & $2.77*10^{-9}$ & 1.610 & 1.178 \\
			\texttt{agg3} & $5.85*10^{-9}$ & 2.121 & 1.195 \\
			\texttt{bandm} & $2.98*10^{-13}$ & 0.191 & 0.474 
		\end{tabular}
	\end{center}
	\caption{Average CPU computation times (s) required to compute iterate $\ve{x}_k$ with $\|A\ve{x}_k-\ve{b}\|_\infty \le 4\|A\ve{x}_{\text{LS}}-\ve{b}\|_\infty$ for the four \texttt{Netlib} problems, \texttt{agg}, \texttt{agg2}, \texttt{agg3}, and \texttt{bandm}.  These values are averaged over 10 trials.}
	\label{tab:Netlib}
\end{table}  

\begin{figure}
	\includegraphics[width=0.5\textwidth]{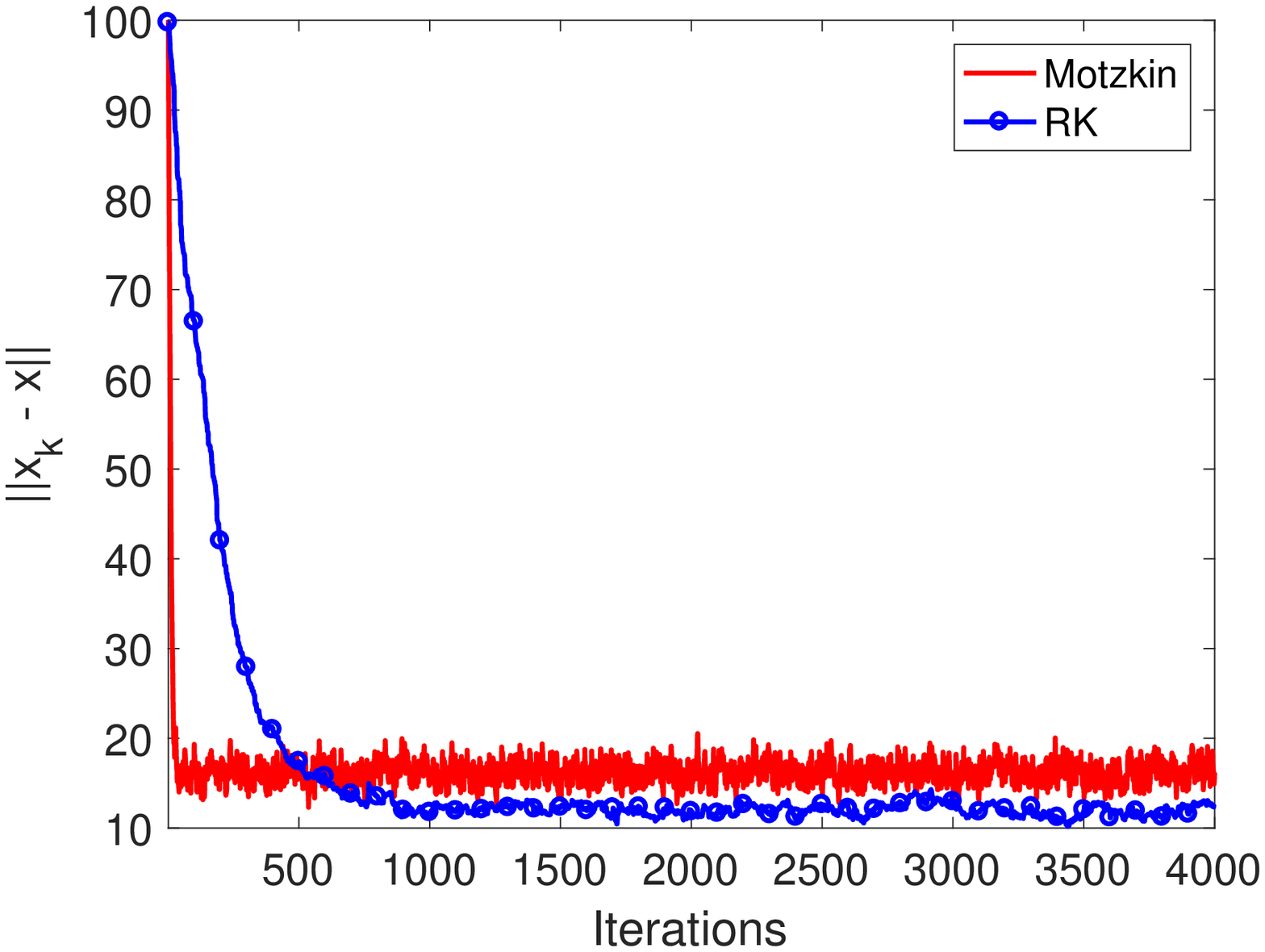}\includegraphics[width=0.5\textwidth]{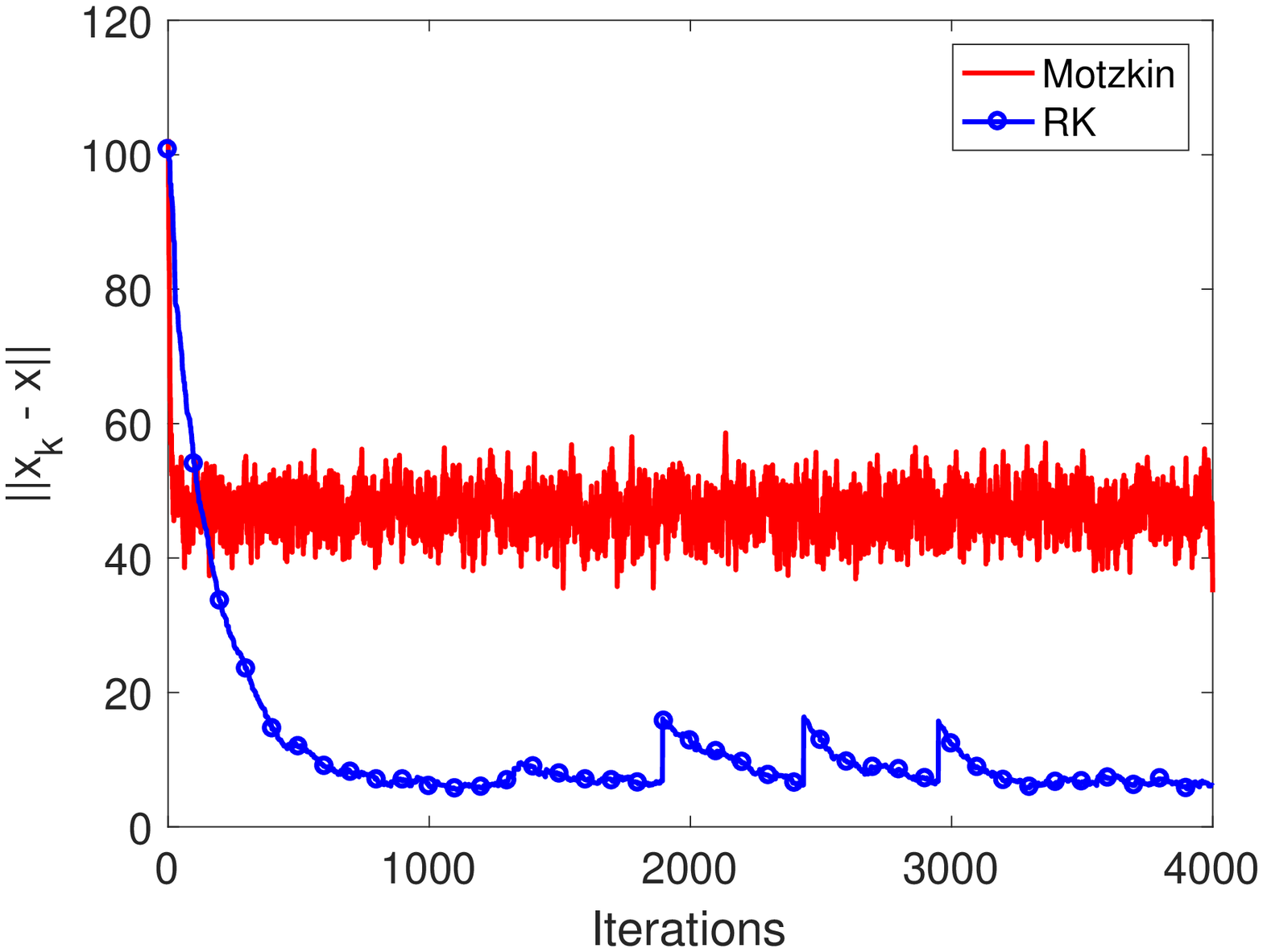}
	\caption{Left: Motzkin's method vs. RK distance from least-squares solution for a Gaussian system with Gaussian noise.  Right: Motzkin's method vs. RK distance from least-squares solution for a Gaussian system with sparse, `spiky' noise.}
	\label{fig:MSGDvsRK}
\end{figure}

To capitalize on this accelerated convergence, one needs knowledge of an upper bound $\|\ve{e}\|_{\infty} \leq \beta$, in which case the stopping criterion of $\|A\ve{x}_k - \ve{b}\|_{\infty} \leq 4\beta$ guarantees the accelerated convergence of \eqref{cor2a} and a final error of $\|\ve{x}_k - \ve{x}\|^2 \leq 25m\sigma_{\min}^{-2}(A)\beta^2$.  Indeed, one quickly verifies that when  $\|A\ve{x}_k - \ve{b}\|_{\infty} \leq 4\beta$, we have
\begin{align*}
\|\ve{x}_k - \ve{x}\| &\leq \sigma_{\min}^{-1}(A)\|A\ve{x}_k - A\ve{x}\|\\
&\leq \sqrt{m}\sigma_{\min}^{-1}(A)\|A\ve{x}_k - A\ve{x}\|_{\infty}\\
&\leq \sqrt{m}\sigma_{\min}^{-1}(A)\left(\|A\ve{x}_k - \ve{b}\|_{\infty} + \|\ve{e}\|_{\infty}\right)\\
&\leq \sqrt{m}\sigma_{\min}^{-1}(A)\left(4\beta + \beta\right).
\end{align*}

\begin{figure}
	\begin{center}
		\includegraphics[width=0.5\textwidth]{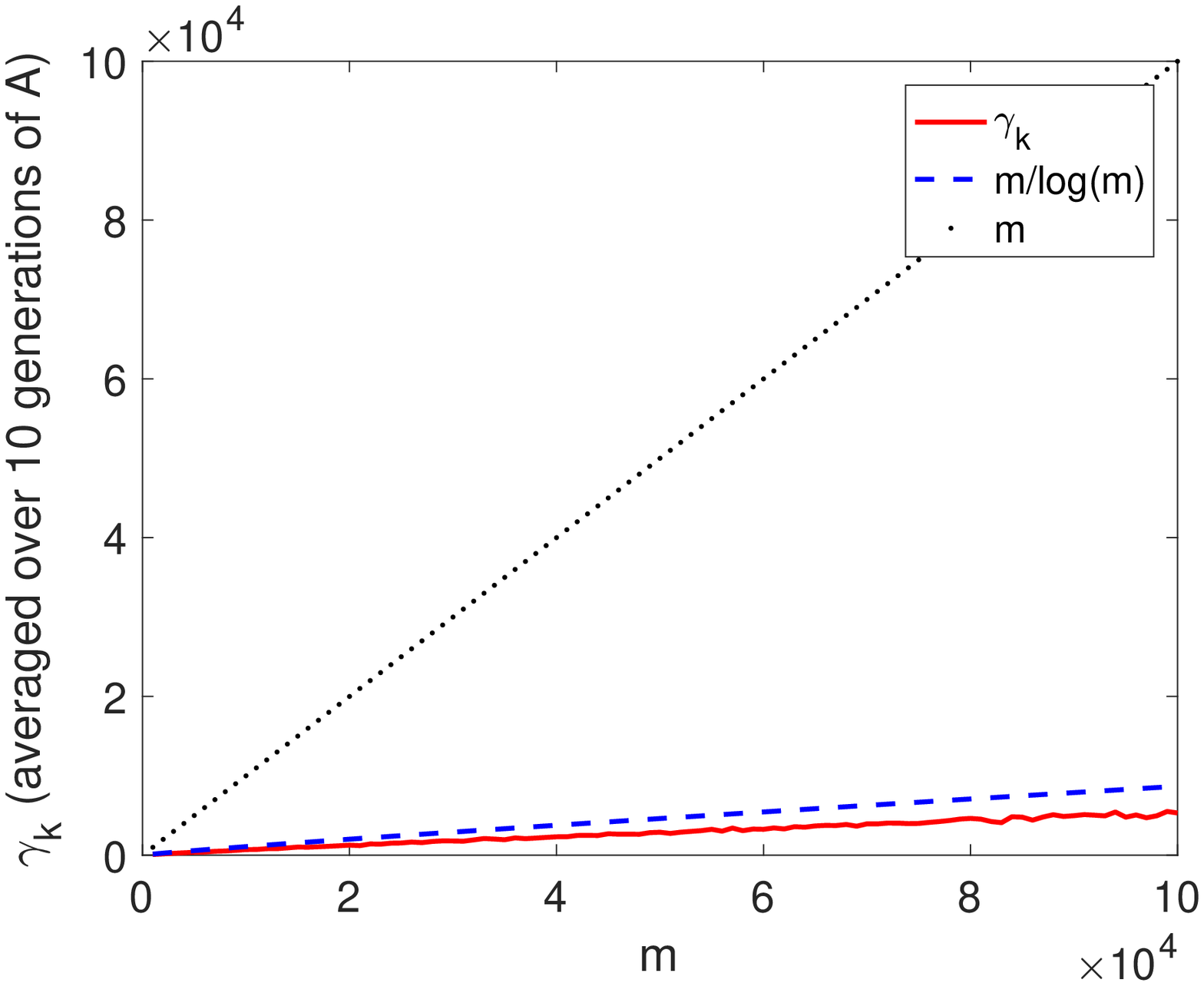}\includegraphics[width=0.5\textwidth]{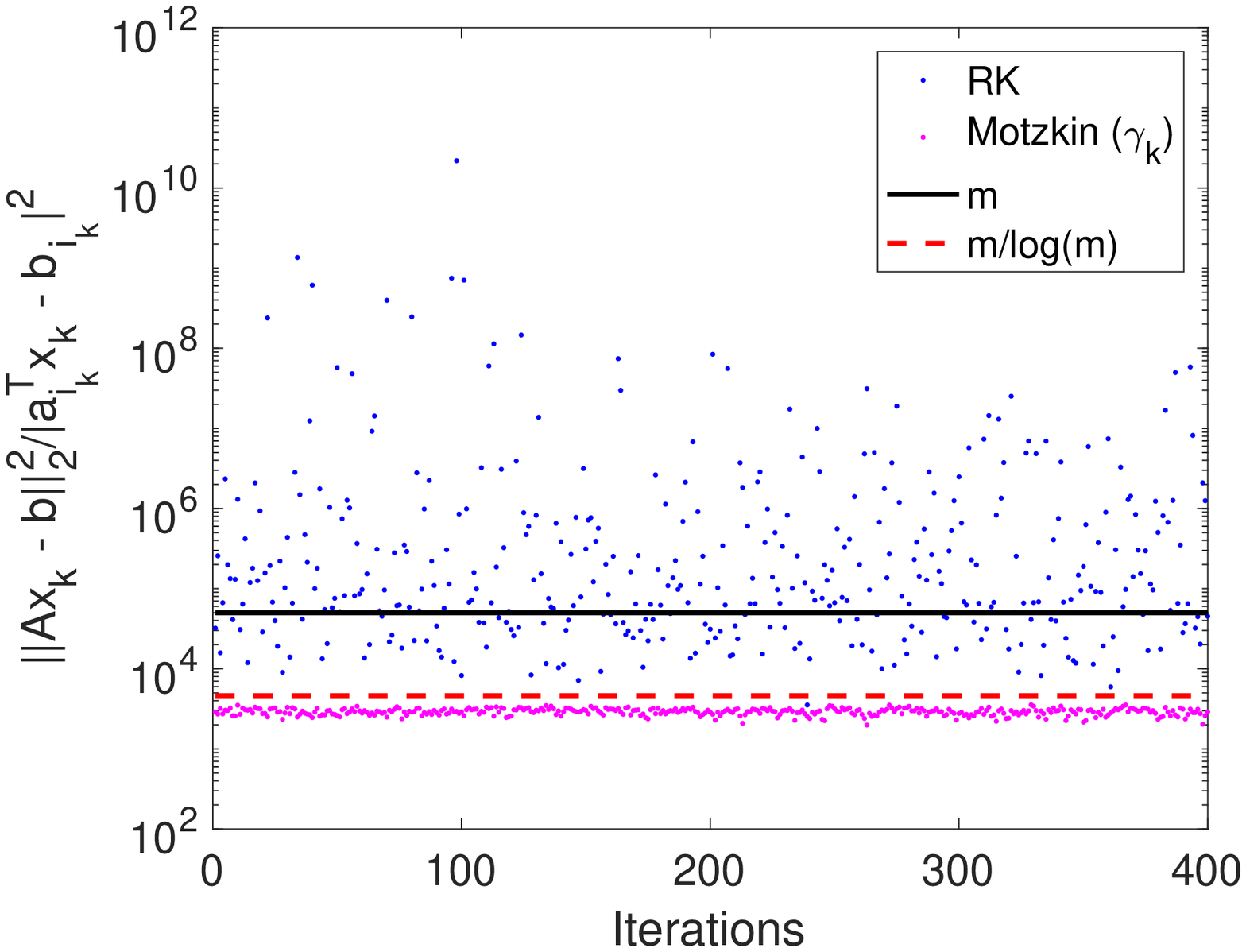}
	\end{center}
	\caption{Left: Average $\gamma_k$ values for various choices of row dimension, $m$, of normalized Gaussian $A \in \R^{m \times 100}$.  Right: An example of the values $\gamma_k$ for a single run of Motzkin's method and the corresponding ratio for RK, the matrix is a $50000 \times 100$ Gaussian. The index $i_k$ denotes the index chosen in the $k$th iteration of each method.  The horizontal lines denote the values ${m}$ and ${m/\log(m)}$.  We see acceleration when $\gamma_k < {m}$.}\label{residualpic}
\end{figure}

Since the acceleration of the method occurs when many of the terms $\gamma_k$ are small, we plot an example in Figure \ref{residualpic}. As expected, many terms are bounded away from $m$. We will analyze this in the Gaussian case further below.

We also only expect this acceleration to be present while the condition of Lemma \ref{lem:motzkin} is in force (i.e. prior to the stopping condition given in the corollary). Once the condition of Lemma \ref{lem:motzkin} is no longer satisfied, selecting greedily will select those entries of the residual which have large contribution from the error, moving the estimation far from the desired solution.  While the difference between greedy selection and randomized selection is not so drastic for Gaussian noise, it will be drastically different for a sparse error.  We include an example system in Figure \ref{fig:spikyerrorgeom} to assist with intuition.  
Again, one could of course implement the Kaczmarz approach after an initial use of the Motzkin method as a strategy to gain acceleration without sacrificing convergence horizon. In Figure \ref{fig:MSGDvsRKwHyb}, we present the convergence of Motzkin's method, the RK method, and a hybrid method which consists of Motzkin iterations until $\|A\ve{x}_k - \ve{b}\|_{\infty} \leq 4\|\ve{e}\|_\infty$, followed by RK iterations.  Again we include results on both a system with Gaussian error and a system with a sparse, `spiky' error, with the systems generated as in Figure \ref{fig:MSGDvsRK}.  

\begin{figure}
	\includegraphics[width=0.5\textwidth]{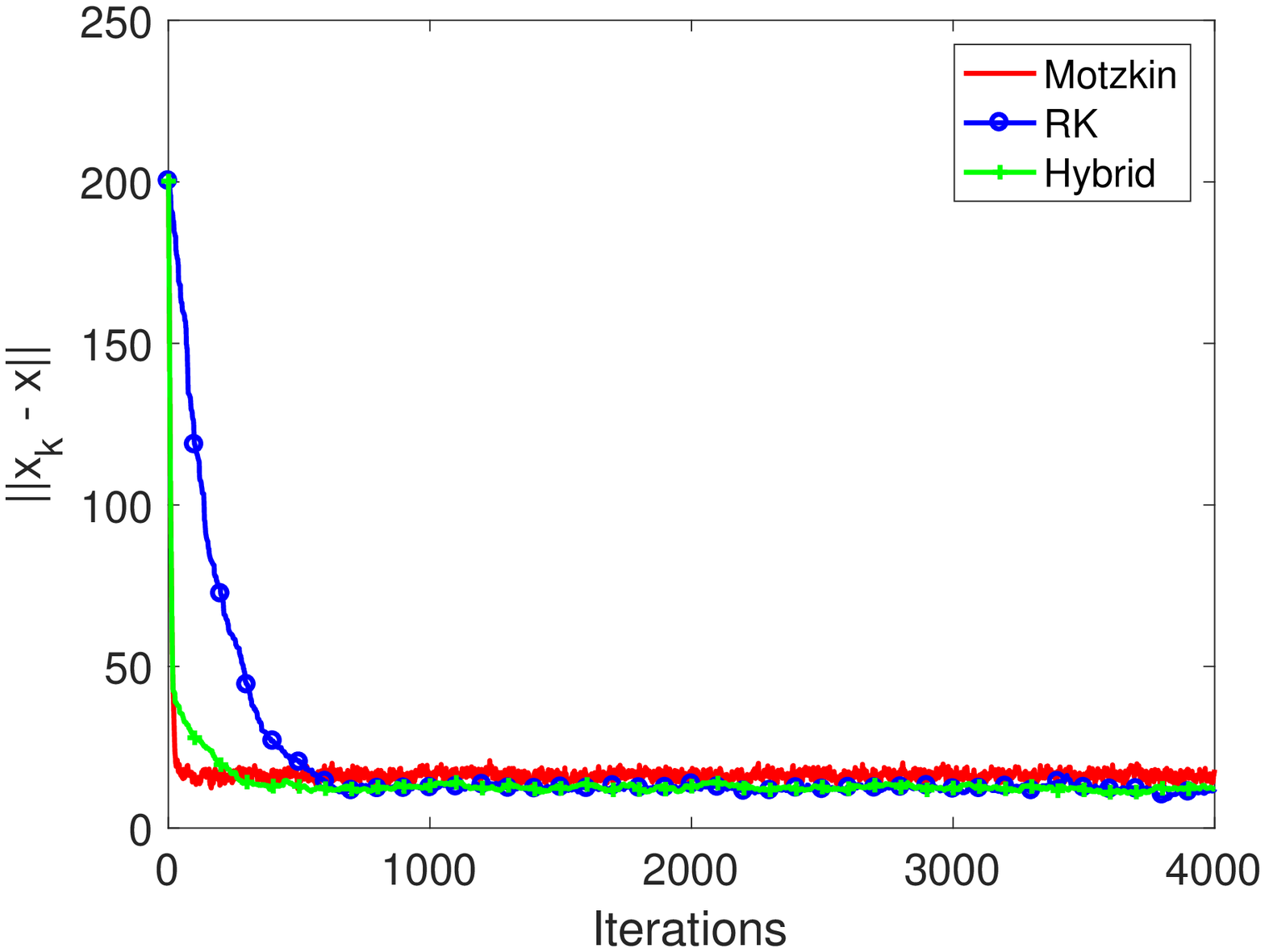}\includegraphics[width=0.5\textwidth]{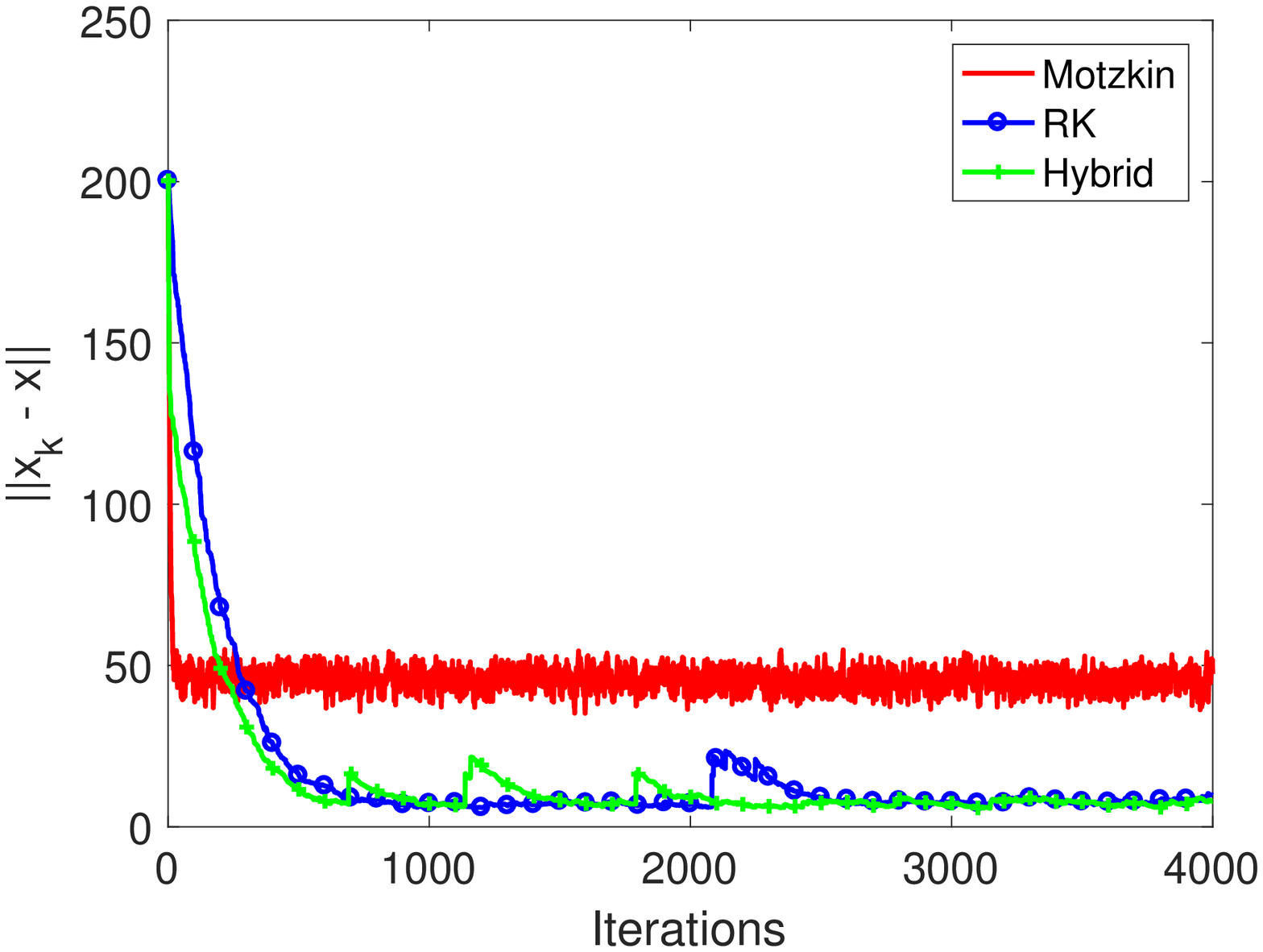}
	\caption{Left: Motzkin's method, RK, and hybrid distance from least-squares solution for a Gaussian system with Gaussian noise.  Right: Motzkin's method, RK, and hybrid distance from least-squares solution for a Gaussian system with sparse, `spiky' noise.}
	\label{fig:MSGDvsRKwHyb}
\end{figure}

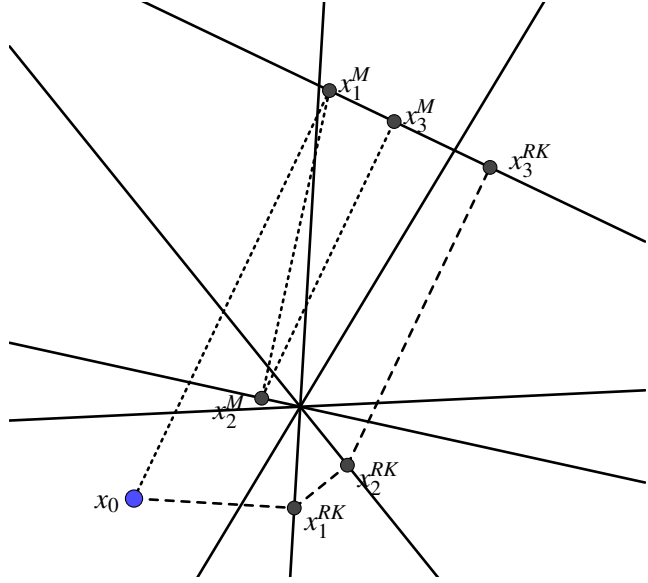
\begin{figure}
	\definecolor{uuuuuu}{rgb}{0.26666666666666666,0.26666666666666666,0.26666666666666666}
	\definecolor{ududff}{rgb}{0.30196078431372547,0.30196078431372547,1.}
	\begin{tikzpicture}[line cap=round,line join=round,>=triangle 45,x=1.0cm,y=1.0cm,scale=1.3]
	\clip(-1.8328966839999432,0.24146467201957236) rectangle (4.6663614276577885,6.121209445477681);
	\draw [line width=1.pt,domain=-1.8328966839999432:4.6663614276577885] plot(\x,{(--0.126--2.46*\x)/1.48});
	\draw [line width=1.pt,domain=-1.8328966839999432:4.6663614276577885] plot(\x,{(--7.0944-0.7*\x)/3.18});
	\draw [line width=1.pt,domain=-1.8328966839999432:4.6663614276577885] plot(\x,{(--2.436-2.38*\x)/-0.14});
	\draw [line width=1.pt,domain=-1.8328966839999432:4.6663614276577885] plot(\x,{(-5.6616-0.14*\x)/-2.94});
	\draw [line width=1.pt,domain=-1.8328966839999432:4.6663614276577885] plot(\x,{(-5.3628--1.96*\x)/-1.58});
	\draw [line width=1.pt,domain=-1.8328966839999432:4.6663614276577885] plot(\x,{(--27.6296-2.24*\x)/4.68});
	\draw [line width=1.pt,dotted] (-0.56,1.04)-- (1.4384261515601784,5.215283209509658);
	\draw [line width=1.pt,dotted] (1.4384261515601784,5.215283209509658)-- (0.7453784228729954,2.066866384902171);
	\draw [line width=1.pt,dotted] (0.7453784228729954,2.066866384902171)-- (2.1006131754430557,4.898338992950333);
	\draw [line width=1.pt,dash pattern=on 3pt off 3pt] (-0.56,1.04)-- (1.0790344827586205,0.9435862068965516);
	\draw [line width=1.pt,dash pattern=on 3pt off 3pt] (1.0790344827586205,0.9435862068965516)-- (1.6223874908869325,1.3815952644693748);
	\draw [line width=1.pt,dash pattern=on 3pt off 3pt] (1.6223874908869325,1.3815952644693748)-- (3.0810157616610923,4.429086473051101);
	\draw (-1.0600783709716528,1.176366618757163) node[anchor=north west] {$x_0$};
	\draw (1.433628190143457,5.555265681885436) node[anchor=north west] {$x_1^M$};
	\draw (0.1465315786857073,2.2013795965056073) node[anchor=north west] {$x_2^M$};
	\draw (2.0981970438484936,5.228613194471096) node[anchor=north west] {$x_3^M$};
	\draw (1.084447944976404,1.0749917088699543) node[anchor=north west] {$x_1^{RK}$};
	\draw (1.6251141310415185,1.514282985047859) node[anchor=north west] {$x_2^{RK}$};
	\draw (3.178154506091514,4.7667941605404785) node[anchor=north west] {$x_3^{RK}$};
	\begin{scriptsize}
	\draw [fill=ududff] (-0.56,1.04) circle (2.5pt);
	\draw [fill=uuuuuu] (1.4384261515601784,5.215283209509658) circle (2.0pt);
	\draw [fill=uuuuuu] (0.7453784228729954,2.066866384902171) circle (2.0pt);
	\draw [fill=uuuuuu] (2.1006131754430557,4.898338992950333) circle (2.0pt);
	\draw [fill=uuuuuu] (1.0790344827586205,0.9435862068965516) circle (2.0pt);
	\draw [fill=uuuuuu] (1.6223874908869325,1.3815952644693748) circle (2.0pt);
	\draw [fill=uuuuuu] (3.0810157616610923,4.429086473051101) circle (2.0pt);
	\end{scriptsize}
	\end{tikzpicture}
	\caption{An example of three iterations of Motzkin's method ($x_k^M$) and three iterations of RK ($x_k^{RK}$) on a Gaussian system with sparse, `spiky' error.  More of the RK iterations are near the least squares solution while Motzkin consistently selects the corrupted equation.}
	\label{fig:spikyerrorgeom}
\end{figure}

\subsection{Heuristics for the Gaussian case}
Here, we study heuristics for our convergence results for the Gaussian matrix case. Note that our results hold for matrices with normalized rows. For simplicity however, we will consider an $m\times n$ matrix whose entries are i.i.d. Gaussian with mean 0 and variance $1/n$. We will then assume we are in the asymptotic regime where this distribution approximates a Gaussian matrix with normalized unit-norm rows\footnote{
	This can be readily verified by observing that the distribution of $\ve{a}_i$ is rotationally invariant and thus $\left(\ve{a}_i^T\frac{\ve{x}}{\|\ve{x}\|}\right)^2$ has the same distribution as $\left(\ve{a}_i^T\ve{e}_1\right)^2$, where $\ve{e}_1$ is the first coordinate vector. Thus it has the same distribution as the ratio of chi-square random variables $g_1^2/\sum_{i=1}^n g_i^2$, for i.i.d. standard normal $g_i$. One then applies Slutsky's theorem to obtain the asymptotic result.} To that end, we assume $m$ and $n$ both grow linearly with respect to one another, and that they are both substantially large.

Define $I_k$ to be the rows of $A$ that are independent from $\ve{x}_k$ and note that $I_k \subseteq I_{k-1} \subseteq ... \subseteq I_1 \subseteq I_0 = [m]$. Fix iteration $k$ and define $m' = m - |I_k|$.  Note that $m-k \le m' \le m$ is the dimension of the sub-matrix whose rows are independent of the iterates up to iteration $k$. 
Throughout this section $\mathbb{P}$ and $\mathbb{E}$ refer to probability and expectation taken with respect to the random and unsampled portion of the matrix $A$, $A_{I_k}$, which has $m'$ rows.

Our first lemma gives a bound on the expected dynamic range for a Gaussian matrix. 

\begin{lemma}\label{lem:gammakbound}
	If $A \in \R^{m \times n}$ is a Gaussian matrix with $a_{ij} \sim \mathcal{N}(0,1/n)$ and $\ve{x}$ is independent of at least $m'$ rows of $A$ (e.g. constructed via $k$ iterations of Motzkin's method) then  
	$$\frac{\mathbb{E}\|A\ve{x}\|^2}{\mathbb{E}\|A\ve{x}\|_\infty^2} \lesssim \frac{n(m' + \sum_{i \not\in I_k} \|\ve{a}_i\|^2)}{\log(m')}.$$
\end{lemma}

\begin{proof}
	First note that 
	\begin{align*}
	\mathbb{E}(\sum_{i=1}^m (\ve{a}_i^T\ve{x})^2) &= \sum_{i=1}^m \mathbb{E} (\ve{a}_i^T\ve{x})^2
	\\&\le \sum_{i=1}^m \mathbb{E}(\|\ve{a}_i\|^2 \|\ve{x}\|^2) & \text{by Cauchy-Schwartz}
	\\&\le \sum_{i \in I_k} \mathbb{E}(\|\ve{a}_i\|^2 \|\ve{x}\|^2) + \sum_{i \not\in I_k} \|\ve{a}_i\|^2 \|\ve{x}\|^2
	\\&= (m' + \sum_{i \not\in I_k} \|\ve{a}_i\|^2) \|\ve{x}\|^2.
	\end{align*}
	
	Next, note that if $\ve{a}_i$ and $\ve{x}$ are independent then $\ve{a}_i^T\ve{x} \sim \mathcal{N}(0,\|\ve{x}\|^2/n)$.  Then
	\begin{align*}
	\mathbb{E}(\max_{i \in [m]} (\ve{a}_i^T\ve{x})^2) &\ge \mathbb{E}(\max_{i \in I_k} (\ve{a}_i^T\ve{x})^2)
	\\& \ge \mathbb{E}(\max_{i \in I_k} \ve{a}_i^T\ve{x})^2
	\\& \ge (\mathbb{E}\max_{i \in I_k} \ve{a}_i^T\ve{x})^2 &\text{by Jensen's inequality}
	\\& \ge \frac{c\|\ve{x}\|^2\log(m')}{n},
	\end{align*}
	as it is commonly known that $\mathbb{E}(\max_{i \in [N]} X_i) \ge c \sigma \sqrt{log N}$ for $X_i \sim \mathcal{N}(0,\sigma^2)$.  Thus, we have $$\mathbb{E}\|A\ve{x}\|_\infty^2 \ge \frac{c \|\ve{x}\|^2 \log(m')}{n} \ge c \frac{\log(m')}{n(m' + \sum_{i \not\in I_k}\|\ve{a}_i\|^2)}\mathbb{E}\|A\ve{x}\|^2.$$
\end{proof}

We can use this lemma along with our main result to obtain the following.

\begin{corollary}
	Let $A \in \R^{m \times n}$ be a normalized Gaussian matrix as described previously, 
	$\ve{x}$ denote the desired solution of the system given by matrix $A$ and right hand side $\ve{b}$, write $\ve{e} = A\ve{x}-\ve{b}$ as the error term and assume $\ve{x}_0$ is chosen so that $\ve{x}_0 - \ve{x}$ is independent of the rows of $A$, $\ve{a}_i^T$. 
	If Algorithm \ref{alg:MotzSelect} is run with stopping criterion $\|A\ve{x}_k - \ve{b}\|_\infty \leq 4\|\ve{e}\|_\infty$, in expectation 
	the method exhibits the accelerated convergence rate
	
	\begin{equation}\label{grrr}
	\mathbb{E} \|\ve{x}_{k+1}-\ve{x}\|^2 \lesssim \mathbb{E}\Bigg[\Bigg(1 - \frac{\log(m') \sigma_{\min}^2(A)}{4nm}\Bigg) \|\ve{x}_k - \ve{x}\|^2 + \frac{1}{2}\|\ve{e}\|_\infty^2\Bigg].
	\end{equation}
\end{corollary}

\begin{proof}
	Beginning from line (\ref{loser}) of the proof of Corollary \ref{cor:Motzrate} and taking expectation of both sides, we have
	\begin{align*}
	\mathbb{E}\|\ve{x}_{k+1}-\ve{x}\|^2 &\le \mathbb{E} \| \ve{x}_k - \ve{x}\|^2 - \frac{1}{4}\mathbb{E} \|A(\ve{x}_k-\ve{x})\|_\infty^2 + \frac{1}{2}\mathbb{E}\|\ve{e}\|_\infty^2
	\\&\lesssim \mathbb{E} \| \ve{x}_k - \ve{x}\|^2 - \frac{\log(m')}{4n(m'+\sum_{i \not\in I_k}\|\ve{a}_i\|^2)}\mathbb{E} \|A(\ve{x}_k-\ve{x})\|^2 + \frac{1}{2}\mathbb{E}\|\ve{e}\|_\infty^2
	\\&= \mathbb{E} \Bigg[\| \ve{x}_k - \ve{x}\|^2 - \frac{ \log(m')}{4nm} \|A\ve{x}_k-A\ve{x}\|^2 + \frac{1}{2}\|\ve{e}\|_\infty^2\Bigg]
	\\&\le \mathbb{E} \Bigg[\Bigg(1 - \frac{\log(m')\sigma_{\min}^2(A)}{4nm}\Bigg) \|\ve{x}_k -\ve{x}\|^2 + \frac{1}{2}\|\ve{e}\|_\infty^2\Bigg]
	\end{align*}
	where the second inequality follows from Lemma \ref{lem:gammakbound} and the fourth from properties of singular values.
\end{proof}

This corollary implies a logarithmic improvement in the convergence rate if $n << \log(m')$, at least initially. Of course, we conjecture that the $\log(m')$ term in \eqref{grrr} is an artifact of the proof and could actually be replaced with $\log(m)$. Additionally, we conjecture that the $n$ in \eqref{grrr} is an artifact of the proof.  This is supported by the experiments shown in Figures \ref{residualpic} and \ref{coolpic}.  Before normalization, the system for the experiment plotted in Figure \ref{coolpic} is defined by Gaussian $A \in \mathbb{R}^{50000 \times 100}$ and $\ve{b} = \ve{e}$ where $\ve{e}$ is a Gaussian vector. Furthermore, Corollary 5.35 of \cite{vershynin2010introduction} provides a lower bound for the size of the smallest singular value of $A$ with high probability, $\mathbb{P}\left(\sigma_{\min}(A) \leq \sqrt{m/n} - 1 - t/\sqrt{n}\right) \leq 2e^{-t^2/2}.$  That is, asymptotically $\sigma_{\min}(A)$ is tightly centered around $\sqrt{m/n} - 1$.

\begin{figure}
	\begin{center}
		\includegraphics[width=0.6\textwidth]{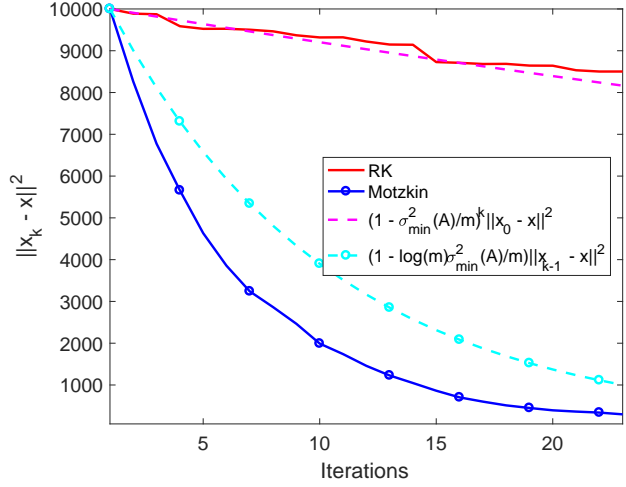}
		\caption{Convergence of Motzkin's method and RK on Gaussian system with corresponding theoretical rate for RK and conjectured rate for Motzkin's method. }\label{coolpic}
	\end{center}
\end{figure}

\section{Conclusion}
We have provided a theoretical analysis for Motzkin's method for inconsistent systems. We show that by using such a greedy selection strategy, the method exhibits an accelerated convergence rate until a particular threshold is reached. This threshold depends on the dynamic range of the residual, and could be estimated to employ a strategy that yields acceleration without sacrificing convergence accuracy. We provide experiments and concrete analysis for Gaussian systems that support our claims. Future work includes a detailed analysis for other types of relevant systems, theoretical guarantees when estimating the residual, and the study of computational tradeoffs as in the framework of \cite{SKM}.  While Motzkin's method can be more computationally expensive than RK, understanding the accelerated convergence per iteration will aid in an analysis of computational tradeoffs for the methods in \cite{SKM}; in addition, it may offer significant advantages in parallel architectures.  These are important directions for future work.

\bibliographystyle{myalpha}     
\bibliography{bib}

\end{document}